\documentclass[12pt]{amsart}
\usepackage{color,graphicx,array, amssymb, amscd}
\newtheorem{Th}{Theorem}
\newtheorem{Theorem}{Theorem}
\numberwithin{Theorem}{subsection} 
\newtheorem{Lemma}[Theorem]{Lemma}
\newtheorem{Proposition}[Theorem]{Proposition}
\newtheorem{Corollary}[Theorem]{Corollary}
\theoremstyle{definition}
\newtheorem{Definition}{Definition}
\numberwithin{Definition}{subsection} 
\theoremstyle{remark}

\newtheorem{Remark}[Theorem]{Remark} 
 
\numberwithin{equation}{section}

\newcommand{\nm}{\nabla^{\mu}}

\newcommand{\sk}{\operatorname {skew}}
\newcommand{\sym}{\operatorname {sym}}

\renewcommand{\l}{\lambda}

\begin{document}
\title{A characterization of the alpha-connections 
on the statistical manifold of normal distributions}
\author[H.~Furuhata]{Hitoshi Furuhata}
\address{Department of Mathematics, Hokkaido University, 
Sapporo, 060-0810, Japan}
\email{furuhata@math.sci.hokudai.ac.jp}
\author[J.~Inoguchi]{Jun-ichi Inoguchi}
\address{Institute of Mathematics, 
University of Tsukuba, 
Tsukuba 305-8571, Japan}
\email{inoguchi@math.tsukuba.ac.jp}
\thanks{The second named author is partially supported by 
Kakenhi 19K03461.}
\author[S.-P.~Kobayashi]{Shimpei Kobayashi}
\address{Department of Mathematics, Hokkaido University, 
Sapporo, 060-0810, Japan}
\email{shimpei@math.sci.hokudai.ac.jp}
\thanks{The third named author is partially supported by Kakenhi 
18K03265.} 
\date{\today} 

\pagestyle{plain}

\begin{abstract} 
We show that the statistical manifold 
of normal distributions is homogeneous.
In particular, it admits 
a $2$-dimensional solvable Lie group structure.
In addition, we
give a geometric characterization
of the Amari-Chentsov $\alpha$-connections on
the Lie group.
\end{abstract}
\maketitle
\section*{Introduction}

The set of normal distributions is parametrized by
$\mathbb{R} \times \mathbb{R}^+$ as 
\begin{equation*}
 \mathbb{R} \times \mathbb{R}^+
  \ni \theta =(\mu, \sigma) \mapsto
  p(t, \theta)=
  \dfrac{1}{\sqrt{2\pi \sigma^2}}
  \exp \left\{
-\dfrac{(t-\mu)^2}{2\sigma^2}
       \right\}, 
\quad t \in \mathbb{R},  
\end{equation*}
where $\mu$ is the mean, and $\sigma^2$ is the variance. 
For tangent vectors $X$, $Y$, $Z$
of an manifold $\mathbb{R} \times \mathbb{R}^+$ at $\theta$,
we define
\begin{equation*}
 g^F(X,Y)=E_\theta[(X\log p) (Y \log p) ], \quad
 C(X, Y, Z)=E_\theta[(X\log p) (Y \log p) (Z \log p)], 
\end{equation*}
where
$\displaystyle 
E_\theta[f]=\int_\mathbb{R}f(t)p(t, \theta)\, dt$ 
for an integrable function $f$ on $\mathbb{R}$.
For a constant $\alpha$, we define $\nabla^{(\alpha)}$ by 
\begin{equation*}
 g^F(\nabla^{(\alpha)}_X Y,Z)
  =g^F(\nabla^{g^F}_X Y, Z)-\dfrac{\alpha}{2} C(X,Y,Z), 
\end{equation*}
where $\nabla^{g^F}$ is the Levi-Civita connection 
of the Riemannian metric $g^F$. 
In information geometry, 
$g^F$ and $\nabla^{(\alpha)}$ are well known 
as the \emph{Fisher metric} and the 
\emph{Amari-Chentsov $\alpha$-connection} 
for the space of normal distributions, respectively. 
In the same fashion, 
we have the Fisher metric and the Amari-Chentsov $\alpha$-connection
for spaces of certain probability densities. 
Abstracting an essence from these, 
we reach the notion of statistical manifolds. 

The Fisher metric and the Amari-Chentsov $\alpha$-connection 
for the space of all the positive probability densities 
on a finite set, that is, the space of multinomial distributions, 
are characterized from a viewpoint of statistics, 
which is known as the Chentsov theorem 
and has been generalized for other spaces, 
see \cite{Dow} for example and references therein. 
The Fisher metric and the Amari-Chentsov $\alpha$-connection 
for the space of normal distributions are expressed as
\begin{gather*}
 g^F=\dfrac{dx^2+2dy^2}{y^2}, \\ 
\nabla^{(\alpha)}_{\partial_x}{\partial_x}=\frac{1-\alpha}{2y}
{\partial_y}, 
\ \ 
\nabla^{(\alpha)}_{{\partial_x}}{\partial_y}=
\nabla^{(\alpha)}_{{\partial_y}}{\partial_x}
=
-\frac{1+\alpha}{y}
{\partial_x},
\ \ 
\nabla^{(\alpha)}_{{\partial_y}}{\partial_y}=-\frac{1+2\alpha}{y}{\partial_y},
\end{gather*}
where 
${\partial_x} ={\partial}/ {\partial x}$,
${\partial_y} ={\partial}/ {\partial y}$ and 
$(x,y)=(\mu, \sigma) \in \mathbb{R}\times\mathbb{R}^+$. 

The goal of this article is to characterize affine connections 
on the statistical manifold of normal distributions
with the Fisher metric 
from a \emph{purely differential geometric viewpoint}, 
especially from a viewpoint of homogeneity. 
For this purpose, in Section \ref{sc:Statistica}, we 
discuss Lie groups equipped with a left-invariant statistical structure, 
called \emph{statistical Lie groups},  
see Definition \ref{def:StatLG}.  
In particular we give an explicit method for constructing 
Lie groups equipped with left-invariant statistical structure. 
In the final section, 
we first show the statistical manifold of normal distributions 
is a statistical Lie group. More precisely 
we will show that the statistical manifold of normal distributions 
admits a solvable Lie group structure with respect to which the statistical 
structure is left-invariant, Proposition \ref{prp:solvable}. 
We finally characterize 
left-invariant connections on 
a $2$-dimensional solvable Lie group with a natural left-invariant metric 
under the total symmetry of the covariant derivative $\nabla C$ of 
 the cubic form $C$, Proposition \ref{prp:Charact}.
As a result, we have the following main theorem. 

\begin{Th}\label{thm:Mainresult}
The set $G$ consisting of all normal distributions admits a Lie group structure 
such that the Fischer metric $g^F=(dx^2+2dy^2)/y^2$ is left-invariant.  
Furthermore, for an affine connection $\nabla$ 
such that $(G, g^F, \nabla)$ is a statistical Lie group, 
the following statements are mutually equivalent{\rm:}
\begin{enumerate}
\item $\nabla$ is the Amari-Chentsov $\alpha$-connection.
\item $\nabla^{g^F} C$ is totally symmetric.
\item $\nabla C$ is totally symmetric.
\item $\nabla^{g^F} K$ is totally symmetric.
\item $R = R^*$. 
\end{enumerate}
\end{Th}

Here $K$ is the skewness operator, and 
$R$ and $R^*$ denote the curvature tensor fields 
of $\nabla$ and its dual connection $\nabla^*$, 
respectively.
The precise definitions are given in Sections \ref{sbsc:statistical} and \ref{subsc:cubic}.


\section{Preliminaries}
In this section we recall the definition of statistical manifolds 
and 
introduce the notion of \textit{homogeneous} statistical manifolds.

\subsection{Statistical manifolds}\label{sbsc:statistical}
Let $M$ be a manifold,
$g$ a Riemannian
metric and 
$\nabla$ a torsion free affine connection.
Then $\nabla$ is said to be 
\textit{compatible} with $g$ if
the covariant derivative 
$\nabla g$ is symmetric. 
A pair $(g,\nabla)$ of a Riemannian metric 
and a compatible affine connection is
called a \emph{statistical structure}
on $M$. A manifold $M$ together with a statistical structure
is called a \emph{statistical manifold}. 
In particular 
a statistical manifold 
is traditionally called a \emph{Hessian manifold} 
if $\nabla$ is flat, 
see for example \cite{Shima}.

A Riemannian manifold $(M,g)$ together
with Levi-Civita connection
$\nabla^{g}$ of $g$ is a typical
example of statistical manifold.
In other words, statistical
manifolds
can be regarded as generalizations
of Riemannian manifolds.

The \emph{conjugate connection} or the \emph{dual connection} $\nabla^*$ of
$\nabla$ with respect to $g$ is 
introduced by the following formula:
\begin{equation}
X g(Y,Z)=
g(\nabla_X Y,Z)+g(Y,\nabla^{*}_{X}Z),
\end{equation}
 where $X,Y,Z \in \mathfrak{X}(M)$.
 Obviously $\nabla=\nabla^*$ if and only if $\nabla$ coincides with the 
Levi-Civita connection $\nabla^{g}$. 

\subsection{The associated skewness operator}\label{subsc:cubic}

On a statistical manifold $(M, g,\nabla)$,
the symmetric tensor field $C$ of type $(0, 3)$,
the so-called \textit{cubic form} on $M$, 
 is defined by 
\begin{equation}\label{eq:cubicform}
C(X,Y,Z)=(\nabla_{X} g)(Y,Z),
\end{equation} 
where $X,Y,Z \in \mathfrak{X}(M)$.
Then we can associate a tensor field $K$ of type $(1, 2)$ by 
\begin{equation}\label{eq:K}
g(K(X,Y),Z)=C(X,Y,Z), 
\end{equation}
where $X,Y,Z \in \mathfrak{X}(M)$.
Furthermore, for every vector field $X$,
one can introduce an endomorphism field $K(X)$ by
\begin{equation*}\label{eq:K2}
K(X)Y=K(X,Y).
\end{equation*}
Since $C$ is totally symmetric,
$K(X)$ is self-adjoint with respect to $g$ and
symmetric, \textit{i.e.}, 
\[
 g(K(X)Y,Z)=g(Y,K(X)Z) \quad \mbox{and} \quad
K(X)Y=K(Y)X.
\]
We call this tensor field $K$ 
the \emph{skewness operator} of
$(M,g,\nabla)$. 
The difference between
$\nabla$ and $\nabla^{g}$ is
given by
\[
\nabla-\nabla^{g}=-\frac{1}{2}K.
\] 
\begin{Remark}
The symbol $K$ is sometimes used as the difference itself,   
for example in \cite{NS}. 
\end{Remark}
The Levi-Civita connection
$\nabla^{g}$ is the ``mean'' of $\nabla$ and 
$\nabla^*$, \textit{i.e.},
$\nabla^{g}=\frac{1}{2}(\nabla+\nabla^*)$.
More generally
for any real number $\alpha$,
\begin{equation}
\nabla^{(\alpha)}_XY =\nabla^{g}_XY-\frac
{\alpha}{2}K(X)Y
\end{equation}
defines a torsion free affine connection $\nabla^{(\alpha)}$.
The affine connection $\nabla^{(\alpha)}$ is called the
$\alpha$-\emph{connection}, \cite{AN}. Note that $\nabla^{(1)}=\nabla$
and $\nabla^{(-1)}=\nabla^*$. 
The covariant derivative of $g$ relative to $\nabla^{(\alpha)}$  is
\[
(\nabla^{(\alpha)}_X g)(Y,Z)=\alpha\, C(X,Y,Z).
\]
Thus $(M, g,\nabla^{(\alpha)})$ is statistical for all
$\alpha \in \mathbb{R}$.
\begin{Remark}
As we have seen before, 
for every statistical manifold
$(M, g, \nabla)$, there exists a
naturally associated symmetric trilinear form $C$.
Conversely let $(M, g, C)$ be a Riemannian 
manifold with a symmetric trilinear form $C$. 
Define a tensor field $K$ by
$ g(K(X)Y,Z)=C(X,Y,Z),$ 
and an affine connection $\nabla$ by
$\nabla=\nabla^{g}-K/2$. 
Then  
the triplet $(M, g,\nabla)$ becomes a statistical manifold.
Thus to equip a statistical structure $(g,\nabla)$ 
is equivalent to equip a pair 
$(g,C)$ consisting of a 
Riemannian metric $g$ and a trilinear form $C$.
In fact, Lauritzen \cite[Chapter 4]{Lau} 
introduced the notion of statistical manifold as a Riemannian 
manifold $(M, g)$ together with a trilinear form $C$.
\end{Remark}

Finally we remark that the curvature tensor field $R$ of $\nabla$ is related to 
the Riemannian curvature $R^{g}=R^{(0)}$ of $g$ by 
\[
R(X,Y)Z=
R^{g}(X,Y)Z+\frac{1}{4}[K(X),K(Y)]Z
-\frac{1}{2}\left(
(\nabla^{g}_{X}K)(Y,Z)-(\nabla^{g}_{Y}K)(X,Z)
\right).
\]
We refer to the readers 
Amari and Nagaoka's textbook \cite{AN}
for general theory of statistical manifolds.

\subsection{Covariant derivative of the cubic form $C$}

We have many choices of affine connection $\nabla$ 
compatible with a prescribed Riemannian metric $g$, 
or equivalently, many choices of cubic forms $C$ on $M$. 
In this section, let us consider a condition
$\nabla C$ is totally symmetric. 
This condition is well known in the affine hypersurface theory. 
In fact, 
a statistical manifold with certain conditions 
can be realized as a Blaschke hypersurface in 
the equiaffine space. 
Then the condition $\nabla C$ is totally symmetric 
means that 
the hypersurface is an affine hypersphere, 
which is the most interesting object in affine differential geometry. 
We refer to the readers \cite{NS} for details. 

\begin{Lemma}[\cite{BNS, Op}]\label{lem:symmetric}
 The following statements are mutually equivalent$:$
\begin{enumerate}
\item $\nabla^{g} C$ is totally symmetric.
\item $\nabla C$ is totally symmetric.
\item $\nabla^{g} K$ is totally symmetric.
\item $R = R^*$. 
\end{enumerate}
\end{Lemma}

Here $R^*$ denotes the curvature tensor field of $\nabla^*$. 
Note that statistical manifolds satisfying the condition $R=R^{*}$ 
have been said to be \emph{conjugate symmetric}, \cite{Lau}. 

The proof of this lemma 
for statistical structures induced by affine hypersurfaces  
still works for our setting. 
We briefly review the proof for the sake of completeness.
 Using the relation $\nabla = \nabla^{g} - \frac{1}{2} K$, we have 
 \[
  (\nabla_X C) (U, V, W) - (\nabla_U C) (X, V, W)
 =(\nabla^{g}_X C) (U, V, W) - (\nabla^{g}_U C) (X, V, W).
 \]
 Thus the conditions $(1)$ and $(2)$ are equivalent.
  Moreover from $C(X, Y, Z) = g (K(X)Y, Z)$, it is easy to see that 
 \[
   (\nabla^{g}_X C) (U, V, W)
 =  g( (\nabla^{g}_X K)(U, V), W).
 \]
 From this, it is easy to see that the conditions $(2)$ and 
 $(3)$ are equivalent. 
 Finally, since
\[
  (\nabla^{g}_X K) (Y, Z)
 - (\nabla^{g}_Y K) (X, Z) 
 = -R(X, Y)Z + R^*(X, Y) Z,
\]
 the conditions $(3)$ and $(4)$ are equivalent.

\subsection{Homogeneous statistical manifolds}
To close this section we introduce the notion of 
homogeneous statistical manifold.
\begin{Definition}\label{def:StatLG}
{\rm Let $(M, g,\nabla)$ be a
statistical manifold and $G$ a Lie group. 
Then $(M, g,\nabla)$ is said 
to be a 
\emph{homogeneous statistical manifold} if 
$G$ acts transitively on $M$ and the action is 
isometric with respect to $g$ and affine with 
respect to $\nabla$. 
}
{\rm In particular, if $G$ is equipped with a 
 statistical structure $(g, \nabla)$ 
such that both  $g$ and $\nabla$ are invariant under left 
translations, then $(G, g, \nabla)$ is a 
homogeneous statistical manifold. The resulting 
homogeneous statistical manifold is called a 
\emph{statistical Lie group}.
}
\end{Definition}
From the next section we will study 
statistical Lie groups.

\section{Left-invariant  connections and left-invariant metrics on Lie groups}
 Let $G$ be a connected 
 Lie group and denote by $\mathfrak{g}$ 
 the Lie algebra of $G$, that is, 
 the tangent space $T_{e}G$ of $G$ at the unit element $ e \in G$. 
 In this section we consider \emph{left-invariant connection}, that is 
affine connections on $G$ which are invariant under left translations by $G$.

Let $\theta$ be the 
Maurer-Cartan form  of $G$. 
By definition, for any tangent vector $X_a$ of $G$ at $a\in G$, we have 
 $\theta _a (X_a)=(dL_{a})^{-1}_{a} X_a\in\mathfrak{g}$. 
 Here $L_a$ denotes the left translation 
 by $a$ in $G$; $L_a : G \ni x \mapsto a x \in G$. 

Hereafter, 
we always assume that Lie groups under consideration will be connected.
\subsection{Left-invariant connections on Lie groups} 
 Take a bilinear map 
 $\mu:\mathfrak{g}\times \mathfrak{g}\to \mathfrak{g}$.
 Then we can define a left-invariant affine connection $\nabla^\mu$ on $G$ by
 its value at the unit element $e \in G$ by
\[
\nabla^{\mu}_{X}Y=\mu(X,Y), \ \ X,Y
\in \mathfrak{g}.
\]
\begin{Proposition}[\cite{Nomizu}]\label{allconnections}
Let $\mathcal{B}_\mathfrak{g}$ be the vector 
 space of all $\mathfrak{g}$-valued 
 bilinear maps on $\mathfrak{g}$ and 
 $\mathcal{A}_G$ the affine space of all 
 left-invariant affine connections on $G$.
 Then the map
\[
\mathcal{B}_{\mathfrak g}\ni \mu\longmapsto \nm \in \mathcal{A}_G 
\] 
 is a bijection between $\mathcal{B}_{\mathfrak g}$ 
 and $\mathcal{A}_G$. The torsion $T^\mu$ of $\nm$ is given by
\[
T^{\mu}(X,Y)=-[X,Y]+\mu(X,Y)-\mu(Y,X)
\]
for all $X$, $Y\in \mathfrak{g}$.
\end{Proposition}

Accordingly, 
$\nabla^{\mu}$ is of torsion free 
if and only if 
\begin{equation}\label{eq:torsionfree}
 \mu(X,Y)-\mu(Y,X)=[X,Y], \ 
i.e., \ 
(\mathrm{skew}\>\mu)(X,Y)=\frac{1}{2}[X,Y], 
\end{equation}
where $(\mathrm{skew}\>\mu)$ is 
the skew symmetric part of $\mu$. 
Thus $\nabla^{\mu}$ has the form:
\begin{equation}\label{eq:nablamu}
\nabla^{\mu}_{X}Y=\frac{1}{2}[X,Y]+(\mathrm{sym}\>\mu)(X,Y), 
\end{equation}
where $(\mathrm{sym}\>\mu)$ is 
the symmetric part of $\mu$.

\subsection{Left-invariant metrics on Lie groups} 
 We equip an inner product $\langle\cdot,\cdot\rangle$ 
 on the Lie algebra $\mathfrak{g}$ of a real Lie group $G$ 
 and extend it to a left-invariant Riemannian metric 
 $g=\langle \cdot,\cdot\rangle$ on $G$.
 Here we define a symmetric bilinear map 
 $U:\mathfrak{g}\times\mathfrak{g}\to \mathfrak{g}$ by
\begin{equation}\label{NR}
 2\langle U(X,Y),Z\rangle 
 =\langle [Z,X], Y\rangle +\langle X,[Z,Y]\rangle
\end{equation}
for $X,Y,Z \in \mathfrak{g}$, see \cite[Chapter X.3.]{KN2}.
This formula implies that 
$g$ is bi-invariant if and only if $U=0$.
 
The Levi-Civita connection $\nabla^{g}$ of 
$G$ is given, as a variant of the Koszul formula, 
by 
\begin{equation}\label{Levi-CivitaRelation}
\nabla^{g}_{X}Y=\frac{1}{2}[X,Y]+U(X,Y)
\end{equation}
for $X, Y \in \mathfrak g$.
 Hence $\nabla^{g}$ is a left-invariant connection 
 $\nabla^{\mu}$ with the bilinear map 
$\mu(X, Y)=\frac{1}{2}[X,Y] +U(X, Y)$. 
Accordingly, we have
\begin{Proposition}
The Levi-Civita connection $\nabla^{g}$ is a 
left-invariant 
connection determined by the bilinear map $\mu$ such that 
\begin{equation}\label{eq:skew-symm}
(\sk \mu)(X,Y)=\frac{1}{2}[X,Y],
\ \ 
(\sym \mu)(X,Y)=U(X,Y).
\end{equation}
\end{Proposition}
\section{Statistical structures on Lie groups}\label{sc:Statistica}

Let $G$ be a Lie group with a 
left-invariant Riemannian metric $g=\langle\cdot,\cdot\rangle$ 
and a left-invariant affine connection $\nabla^{\mu}$. 
The covariant derivative $C=\nabla^{\mu} g$ 
is computed as:
\begin{align*}
C(X,Y,Z)
=-\langle \mu(X,Y),Z\rangle
-\langle Y,\mu(X,Z)\rangle.
\end{align*}
Since we have  
$C(Y,X,Z)=-\langle \mu(Y,X),Z\rangle
-\langle X,\mu(Y,Z)\rangle$ and 
$C(Y,Z,X)=-\langle \mu(Y,Z),X\rangle
-\langle Z,\mu(Y,X)\rangle$ 
analogously, 
the total symmetry condition of $C$ is
\[
\langle \mu(X,Y)-\mu(Y,X),Z\rangle
=
\langle X,\mu(Y,Z)\rangle
-\langle Y, \mu(X,Z)\rangle.
\]
 Then by using the torsion free condition in \eqref{eq:torsionfree} and 
 the form of $\mu$ in \eqref{eq:nablamu}, 
 we have 
\begin{align*}
\langle [X,Y],Z\rangle
=
\langle X,(\mathrm{sym}\>\mu)(Y,Z)
\rangle
-\langle Y, (\mathrm{sym}\>\mu)(X,Z)
\rangle
+\frac{1}{2}
\langle X, [Y, Z] \rangle
- \frac{1}{2}\langle Y, [X, Z]
\rangle.
\end{align*}
Finally using the definition of $U(X, Y)$ in \eqref{NR}, we have 
\begin{equation}\label{leftinvariant}
\langle U(Y,Z),X\rangle
-\langle U(X,Z),Y\rangle
=\langle (\mathrm{sym}\>\mu)(Y,Z),X
\rangle
-\langle (\mathrm{sym}\>\mu)(X,Z),Y
\rangle.
\end{equation}
 Thus when $g$ is bi-invariant, 
 by \eqref{eq:skew-symm} the total symmetry condition is 
\[
\langle X, U(Y,Z)
\rangle
=\langle Y, U(X,Z)
\rangle.
\]
Now let us put $\nu:=\mathrm{sym}\>\mu$ then 
we obtain the following recipe for
constructing statistical Lie groups.

\begin{Proposition}
Let $G$ a Lie group equipped with a left-invariant metric $g$ and 
the symmetric bilinear map $U$ defined in \eqref{NR}. 
Then every left-invariant connection $\nabla$ compatible with $g$ 
is represented as 
\[
\nabla_{X}Y=\nu(X,Y)+\frac{1}{2}[X,Y]
\]
for some symmetric bilinear map 
$\nu:\mathfrak{g}\times\mathfrak{g}\to\mathfrak{g}$ satisfying
\[
\langle U(Y,Z) -\nu(Y,Z),X\rangle
=
\langle Y,  U(X,Z)- \nu(X,Z)\rangle.
\]
\end{Proposition}

\begin{Corollary}
 Let $G$ be a Lie group equipped with a 
 left-invariant statistical structure $(g,\nabla)$.
 Then the following statements hold$:$
\begin{enumerate}
\item 
The connection $\nabla$ is 
represented as 
\[
\nabla_{X}Y=\frac{1}{2}[X,Y] + U(X,Y)-\frac{1}{2}K(X)Y  
\]
for a self-adjoint symmetric bilinear map $K$, \textit{i.e.}, 
$K:\mathfrak{g}\times\mathfrak{g}\to
\mathfrak{g}$ satisfying
\[
\langle K(X)Z,Y\rangle
=\langle Z,K(X)Y\rangle, \quad K(X)Y=K(Y)X.
\]

\item 
Furthermore, 
if the left-invariant statistical structure 
is  bi-invariant, 
then the connection is expressed as
\[
\nabla_XY=\frac{1}{2}[X,Y]-\frac{1}{2}K(X)Y 
\]
for 
an $\mathrm{Ad}(G)$-invariant self-adjoint symmetric bilinear map 
$K$. 
\end{enumerate}
\end{Corollary}

We remark that one can reformulate the above corollary 
by using $C$, that is, a trilinear map on $\mathfrak g$, 
instead of $K$. 

\section{The statistical manifold of normal distributions}

In this section we prove that the statistical manifold 
of normal distributions is homogeneous. 
More precisely, we prove that the statistical manifold of normal distributions is 
identified with a $2$-dimensional solvable Lie group equipped with a 
 left-invariant 
statistical structure, Proposition \ref{prp:solvable}. Next, 
we characterize the Amari-Chentsov $\alpha$-connection 
on the statistical manifold of normal distributions 
in terms of the covariant derivative of the cubic form $\nabla C$.

\subsection{Two dimensional Lie groups}
In this subsection we give some explicit examples of $2$-dimensional 
statistical Lie groups.
It is known \cite{Bry} that every $2$-dimensional Lie group is
either
abelian or non-abelian and isomorphic to 
\[
\mathfrak{g}=\left\{
\left(
\begin{array}{cc}
v & u\\
0 & 0
\end{array}
\right)
\ 
\biggr \vert \ u,v\in
\mathbb{R}
\right\}.
\]
The Lie algebra $\mathfrak{g}$ is generated by 
\[
E_1=\left(
\begin{array}{cc}
0 & 1\\
0 & 0
\end{array}
\right) \ \ \mbox{and}\ \ 
E_2=\left(
\begin{array}{cc}
1 & 0\\
0 & 0
\end{array}
\right)
\] 
with commutation relation 
$[E_1,E_2]=-E_1$.
The simply connected and connected Lie group $G$ 
corresponding to $\mathfrak{g}$ is 
\begin{equation}\label{eq:two-G}
G=\left\{
\left(
\begin{array}{cc}
y & x\\
0 & 1
\end{array}
\right)\>\biggr|\>x,y\in
\mathbb{R},\ \ y>0
\right\}
\end{equation}
with global coordinate system $(x,y)$. 
Note that $G$ is the Lie group of orientation 
preserving affine transformations of the real line $\mathbb{R}$. We denote by 
$\psi$, the inverse map of this global coordinate system. 
Then a direct computation shows that  
$\psi :\mathbb R \times \mathbb R_+ \to G$ satisfies
 \[
  \psi(x, y)^{-1} \partial_x  \psi(x, y)  = y^{-1} E_1 \quad \mbox{and}\quad 
 \psi(x, y)^{-1} \partial_y  \psi(x, y)  = y^{-1} E_2.
 \]
Thus for a constant $\lambda>0$, we put 
\[
\{e_1:=E_1,\ \ e_2=\lambda^{-1}E_2\}.
\]
Note that $[e_1,e_2]=-\lambda^{-1}e_1$.

We introduce a left-invariant Riemannian metric $g$ 
for which $\{e_1,e_2\}$ is orthonormal.
Then the induced left-invariant metric $g$ is 
\[
g=\frac{dx^2+\lambda^2dy^2}{y^2}.
\]
When $\lambda=1$, the resulting Riemannian manifold 
is the hyperbolic plane of curvature $-1$. 
In case $\lambda=\sqrt{2}$, the metric $g$ is the Fisher metric of 
the space of the normal distributions.

The symmetric bilinear map $U$ in \eqref{NR} is computed as
$2\langle U(e_1,e_1),e_1\rangle =0$ and 
\[
 2\langle U(e_1,e_1),e_2\rangle =
\langle e_1, [e_2, e_1]\rangle +\langle e_1, [e_2, e_1]\rangle =
2\lambda^{-1}.
\]
Thus $U(e_1,e_1)=\lambda^{-1}e_2$.
Similarly 
we have $U(e_1,e_2)=-\lambda^{-1}e_1/2$ and $U(e_2,e_2)=0$.
 The Levi-Civita connection is described as
\[
\nabla^{g}_{e_1}e_{1}=\frac{1}{\lambda}e_2,\ \ 
\nabla^{g}_{e_1}e_{2}=-\frac{1}{\lambda}e_{1},
\ \ 
\nabla^{g}_{e_2}e_{1}=0,\ \ 
\nabla^{g}_{e_2}e_{2}=0.
\]
The Amari-Chentsov
$\alpha$-connection of the statistical manifold 
of the normal distributions is naturally extended to the 
following one-parameter family 
$\{\nabla^{(\alpha)}\>|\>\alpha\in\mathbb{R}\}$ 
of connections on $G$:
\[
\nabla^{(\alpha)} _{\partial_x}{\partial_x}
=\frac{1-\alpha}{\lambda^2 y}{\partial_y},
\ \ 
\nabla^{(\alpha)} _{{\partial_x}}{\partial_y}=
\nabla^{(\alpha)} _{{\partial_y}}{\partial_x}
=
-\frac{1+\alpha}{y}{\partial_x},
\ \ 
\nabla^{(\alpha)} _{{\partial_y}}{\partial_y}
=-\frac{1+2\alpha}{y}{\partial_y}. 
\]
It should be remarked that every connection $\nabla^{(\alpha)}$ is left-invariant. In fact 
the table of covariant derivatives 
with respect to $\nabla^{(\alpha)}$ is 
rephrased as
\begin{equation}\label{eq:table}
\nabla^{(\alpha)} _{e_1}e_1=
\frac{1-\alpha}{\lambda}e_2,
\  \
\nabla^{(\alpha)} _{e_1}e_2=-
\frac{1+\alpha}{\lambda}e_1,
\ \ 
\nabla^{(\alpha)} _{e_2}e_1=-
\frac{\alpha}{\lambda}e_1,
\ \ 
\nabla^{(\alpha)} _{e_2}e_2=-
\frac{2\alpha}{\lambda}e_2.
\end{equation}
\begin{Proposition}\label{prp:solvable}
The statistical manifold 
$\left(
\{(x,y)\in\mathbb{R}^2\>|\>y>0\}, g, 
\nabla^{(\alpha)} \right)$ of the normal distributions 
mentioned in Introduction 
is identified with  the solvable Lie group $G$ in \eqref{eq:two-G}
 equipped with a left-invariant metric  $g=(dx^2+2dy^2)/y^2$ and a left-invariant 
 connection $\nabla^{(\alpha)}$ in \eqref{eq:table} with $\lambda=\sqrt{2}$. 
Furthermore, the skewness operator is given by 
\[
K(e_1,e_1)=\sqrt{2}e_2,
\ \ 
K(e_1,e_2)=\sqrt{2}e_1,
\ \ 
K(e_2,e_2)=2\sqrt{2}e_2.
\]
\end{Proposition}

\begin{Remark}
Under the identification of the statistical manifold of normal distributions 
with the solvable Lie group $G$ given by \eqref{eq:two-G}, 
$G$ acts on the real line $\mathbb{R}$ as affine transformations. 
One can see that 
$(\mu,\sigma)^{-1}: \mathbb{R} \ni t \mapsto (t-\mu)/\sigma \in \mathbb{R}$.  
This implies a transformation of random variables. 
Every normal distribution $(\mu,\sigma)\in G$ is 
translated to the standard normal distribution $(0,1)$ 
by left translation $(\mu,\sigma)^{-1}$. 
This is nothing but the standardized form 
of a random variable of $(\mu,\sigma)$.
\end{Remark}

\subsection{A characterization of the $\alpha$-connections}
Finally we give a characterization of the $\alpha$-connections 
on the statistical manifold of normal distributions in terms 
of the total symmetry of $\nabla C$. 
\begin{Proposition}\label{prp:Charact}
 Let $G$ be a $2$-dimensional solvable Lie group defined in \eqref{eq:two-G}
 with a left-invariant metric $g = (d x^2  + \l^2 d y^2)/y^2$. If 
 $G$ admits a left-invariant statistical structure $(G, g, \nabla)$ satisfying 
 one of the conditions in Lemma $\ref{lem:symmetric}$. 
 Then $\nabla$ is 
 given by 
\begin{align*}
\nabla_{e_1} e_1 & =\left(\frac{1}{\lambda}- \frac{p}{2} \right)e_2, \quad 
\nabla_{e_1} e_2 =\left(- \frac{1}{\lambda}- \frac{p}{2} \right) e_1, 
\\
\nabla_{e_2} e_1 &= - \frac{p}{2} e_1, \quad
\nabla_{e_2} e_2 = -p e_2, 
\end{align*}
 where $p \in \mathbb R$ and $\{e_1, e_2\}$ is 
the left-invariant orthonormal frame field as before.
\end{Proposition}
\begin{proof}
 The condition $(3)$ in Lemma \ref{lem:symmetric} is equivalent to
\[
 (\nabla^{g}_{e_1} K) (e_2, e_2) = 
 (\nabla^{g}_{e_2} K) (e_1, e_2), \quad 
 (\nabla^{g}_{e_1} K) (e_2, e_1) = 
 (\nabla^{g}_{e_2} K) (e_1, e_1).
\] 
 Since $\nabla^{g}_{e_2} e_j=0 \;(j=1, 2)$, we can compute the above equations
 as
\begin{gather*}
K_{22}^1 e_2 -K_{22}^2 e_1 + 2 \sum_{\ell=1}^{2} K_{12}^{\ell} e_{\ell}=0, \\
K_{21}^{1} e_2 -K_{21}^2 e_1 +\sum_{\ell=1}^{2} K_{11}^{\ell} e_{\ell}- 
 \sum_{\ell=1}^{2} K_{22}^{\ell} e_{\ell}=0.
\end{gather*}
Here we write $K(e_i, e_j) = \sum_{\ell=1}^2 K_{ij}^{\ell} e_{\ell}$. 
 Thus the skewness operator $K$ can be explicitly given as
\begin{equation}\label{eq:skewness}
 K(e_1, e_1) = p e_2, \quad 
 K(e_1, e_2)= K(e_2, e_1) = p e_1,  \quad 
 K(e_2, e_2) = 2 p e_2.
\end{equation}
 By using the relation $\nabla = \nabla^{g}- \frac{1}{2} K$,
 the  claim follows.
\end{proof}

By setting $p = 2\lambda^{-1} \alpha$ with $\lambda = \sqrt{2}$ and 
 by using Lemma \ref{lem:symmetric}, we arrive at Theorem \ref{thm:Mainresult}.

\bibliographystyle{plain}

\begin{thebibliography}{10}
\bibitem{AN}
S.~Amari, K.~Nagaoka,
\textit{Method of Information Geometry},
Amer. Math. Soc., Oxford Univ. Press, 2000.

\bibitem{BNS}
N.~Bokan, K.~Nomizu, U.~Simon,
\newblock Affine hypersurfaces with parallel cubic forms.
 {\em Tohoku Math. J. (2)}  \textbf{42} (1990), no. 1, 101--108. 

\bibitem{Bry}
 R.~L.~Bryant, 
 \textit{An introduction to Lie groups and symplectic geometry. Geometry and quantum field theory},
 IAS/Park City Math. Ser., \textbf{1} (1995), 5--181.

\bibitem{Dow}
 J.~G.~Dowty,
\newblock Chentsov's theorem for exponential families, 
{\em Information Geometry} \textbf{1} (2018), no. 1, 117--135.


\bibitem{KN2}
 S.~Kobayashi, K.~Nomizu, 
 \newblock {\em Foundations of Differential Geometry} II, 
 Interscience Tracts in Pure and Applied Math. \textbf{15}, 
 \newblock Interscience Publishers, 1969. 


\bibitem{Lau}
S.~L.~Lauritzen,  Statistical manifolds, in: 
\newblock {\em Differential Geometry in Statistical Inference},
IMS Lecture Notes: Monograph Series, \textbf{10}, 
 Institute of Mathematical Statistics, Hayward, 
 California, 1987, pp.~163--216.

\bibitem{Nomizu}
 K.~Nomizu, 
 \newblock Invariant affine connections on homogeneous spaces,
 \newblock {\em Amer. J. Math.} \textbf{76} (1954), no.~1, 33--65.

 \bibitem{NS} K.~Nomizu, 
T.~Sasaki, {\em Affine Differential Geometry}, 
Cambridge Univ. Press, 1994. 


\bibitem{Op} 
 B.~Opozda, 
\newblock A sectional curvature for statistical structures,
{\em Linear Algebra Appl.} \textbf{497} (2016), 134--161. 


\bibitem{Shima}
H.~Shima, 
\textit{The Geometry of Hessian Structures}, 
World Scientific, 2007.
\end{thebibliography}
\def\cprime{$'$}

\end{document}